\documentclass[reqno]{amsart}
\usepackage{amssymb,setspace}
\usepackage{ifpdf}
\ifpdf
 \usepackage[hyperindex,pagebackref]{hyperref}%
\else
 \expandafter\ifx\csname dvipdfm\endcsname\relax
 \usepackage[hypertex,hyperindex,pagebackref]{hyperref}
 \else
 \usepackage[dvipdfm,hyperindex,pagebackref]{hyperref}
 \fi
\fi
\theoremstyle{plain}
\newtheorem{theorem}{Theorem}[section]
\newtheorem{lemma}{Lemma}[section]
\newtheorem{corollary}{Corollary}[section]
\theoremstyle{remark}
\newtheorem{remark}{Remark}[section]
\DeclareMathOperator{\td}{d}
\numberwithin{equation}{section}
\allowdisplaybreaks[4]

\begin{document}

\title[Bounds for combinations of Toader mean and arithmetic mean]
{Bounds for the combination of Toader mean and the arithmetic mean in terms of the contraharmonic mean}

\author[W.-D. Jiang]{Wei-Dong Jiang}
\address[Wei-Dong Jiang]{Department of Information Engineering, Weihai Vocational College, Weihai City, Shandong Province, 264210, China}
\email{\href{mailto: W.-D. Jiang <jackjwd@hotmail.com>}{jackjwd@hotmail.com}}

\author[F. Qi]{Feng Qi}
\address[Feng Qi]{Department of Mathematics, College of Science, Tianjin Polytechnic University, Tianjin City, 300387, China}
\email{\href{mailto: F. Qi <qifeng618@gmail.com>}{qifeng618@gmail.com}, \href{mailto: F. Qi <qifeng618@hotmail.com>}{qifeng618@hotmail.com}}
\urladdr{\url{http://qifeng618.wordpress.com}}

\begin{abstract}
In the paper, the authors find the greatest value $\lambda $ and the least value $\mu $ such that the double inequality
\begin{multline*}
C(\lambda a+(1-\lambda)b,\lambda b+(1-\lambda )a)<\alpha A(a,b)+(1-\alpha)T(a,b)\\
< C(\mu a+(1-\mu)b,\mu b+(1-\mu )a)
\end{multline*}
holds for all $\alpha\in(0,1)$ and $a,b>0$ with $a\ne b$, where
$$
C(a,b)=\frac{a^{2}+b^{2}}{a+b},\quad A(a,b)=\frac{a+b}2,
$$
and
$$
T(a,b)=\frac{2}{\pi}\int_{0}^{{\pi}/{2}}\sqrt{a^2\cos^2\theta+b^2\sin^2\theta}\,\td\theta
$$
denote respectively the contraharmonic, arithmetic, and Toader means of two positive
numbers $a$ and $b$.
\end{abstract}



\thanks{This paper was typeset using \AmS-\LaTeX}

\maketitle

\section{Introduction}
For $p\in \mathbb{R}$ and $a,b>0$, the contraharmonic mean $C(a,b)$, the $p$-th power mean $M_{p}(a,b)$, and Toader mean $T(a,b)$ are respectively defined by
$$
C(a,b)=\frac{a^{2}+b^{2}}{a+b},\quad
M_{p}(a,b)=
\begin{cases}
\biggl(\dfrac{a^{p}+a^{p}}{2}\biggr)^{{1}/{p}},&p\ne 0,\\
\sqrt{ab}\,,&p=0,
\end{cases}
$$
and
\begin{equation}\label{eq1.1}
T(a,b)=\dfrac{2}{\pi}\int_{0}^{\pi/2}\sqrt{a^2\cos^2\theta+b^2\sin^2\theta}\,\td\theta\\
=
\begin{cases}
\dfrac{2a}\pi\mathcal{E}\Biggl(\sqrt{1-\biggl(\dfrac{b}a\biggr)^2}\,\Biggr),&a>b,\\
\dfrac{2b}\pi\mathcal{E}\Biggl(\sqrt{1-\biggl(\dfrac{b}a\biggr)^2}\,\Biggr),&a<b,\\
a,&a=b,
\end{cases}
\end{equation}
where
$$
{\mathcal{E}}={\mathcal{E}}(r)=\int_{0}^{\pi/2}\sqrt{1-r^2\sin^2\theta}\,\td\theta
$$
for $r\in[0,1]$ is the complete elliptic integral of the second kind. For more information on complete elliptic integrals, please see~\cite{Guo-Qi-MIA-11, Qi-Cui-Xu-MIA-99, Qi-Huang-Tamkang-98, Qi-Niu-Guo-JIA-09} and plenty of references therein.
\par
Recently, Toader mean has attracted attention of several researchers. In particular, many remarkable inequalities for $T(a,b)$ can be found in the literature~\cite{Chu-Miao-Kun-AAA-12, Chu-Miao-Kun-Results-Math-12, Chu-Wang-Qiu-Indian-Acad-12, Chu-Miao-Kun-2Qiu-AAA-11, Wang-Chu-Qiu-Jiang-JAT-12}.
It was conjectured in~\cite{Vuorinen-Madras-98} that
\begin{equation}\label{(1.3)-eq}
M_{{3}/{2}}(a,b)<T(a,b)
\end{equation}
for all $a,b>0$ with $a\ne b$. This conjecture was proved in~\cite{Barnard-Pearce-Richards-SIAM-JMA-00, Qiu-Shen-Hangzhou-97} respectively.
In~\cite{Alzer-Qiu-JCAM-04}, a best possible upper bound for Toader mean was presented by
\begin{equation}
T(a,b)<M_{{\ln{2}}/{\ln{({\pi}/{2})}}}(a,b)
\end{equation}
for all $a,b>0$ with $a\ne b$.
\par
It is not difficult to verify that
\begin{equation}\label{(1.5)-eq}
C(a,b)>M_{2}(a,b)=\sqrt{\frac{a^{2}+b^{2}}{2}}\,
\end{equation}
for all $a,b>0$ with $a\ne b$.
From~\eqref{(1.3)-eq} to~\eqref{(1.5)-eq} one has
$$
A(a,b)<T(a,b)<C(a,b)
$$
for all $a,b>0$ with $a\ne b$.
\par
For positive numbers $a, b>0$ with $a\ne b$, let
\begin{equation}\label{J(x)-dfn-eq}
J(x)=C\bigl(xa+(1-x)b,xb+(1-x)a\bigr)
\end{equation}
on $\bigl[\frac{1}{2},1\bigr]$. It is not difficult to verify that $J(x)$ is continuous and strictly increasing on $\bigl[\frac{1}{2},1\bigr]$. Note that $J\bigl(\frac12\bigr)=A(a,b)<T(a,b)$ and $J(1)=C(a,b)>T(a,b)$.
\par
In~\cite{Chu-MK-Ma-JMI-13} it was proved that the double inequality
\begin{equation}\label{eq1.6}
 C\bigl(\alpha a+(1-\alpha)b,\alpha b+(1-\alpha)a\bigr)<T(a,b)<C\bigl(\beta a+(1-\beta)b,\beta b+(1-\beta)a\bigr)
\end{equation}
holds for all $a,b>0$ with $a\ne b$ if and only if $\alpha\le \frac34$ and $\beta\ge \frac12+\frac{\sqrt{4\pi-\pi^2}\,}{2\pi}$.
\par
The main purpose of the paper is to find the greatest value $\lambda $ and the least value $\mu $ such that the double inequality
\begin{multline}
C(\lambda a+(1-\lambda )b,\lambda b+(1-\lambda )a)<\alpha A(a,b)+(1-\alpha)T(a,b)\\
< C(\mu a+(1-\mu )b,\mu b+(1-\mu )a)
\end{multline}
holds for all $\alpha\in(0,1)$ and $a,b>0$ with $a\ne b$. As applications, we also present new bounds for the complete
elliptic integral of the second kind.

\section{Preliminaries and Lemmas}
In order to establish our main result, we need several formulas and lemmas below.
\par
For $0<r<1$ and $r'=\sqrt{1-r^2}\,$, Legendre's complete elliptic integrals of the first and second kinds are defined in~\cite{Bowman-Dover-61, Byrd-Friedman-Springer-71} by
\begin{equation*}
\begin{cases}\displaystyle
{\mathcal{K}}={\mathcal{K}}(r)=\int_{0}^{{\pi}/{2}}\frac1{\bigl(1-r^2\sin^2\theta\bigr)^{1/2}}\td\theta,\\
{\mathcal{K}}'={\mathcal{K}}'(r)={\mathcal{K}}(r'),\\
{\mathcal{K}}(0)=\dfrac\pi2,\\
{\mathcal{K}}(1)=\infty
\end{cases}
\end{equation*}
and
\begin{equation*}
\begin{cases}\displaystyle
{\mathcal{E}}={\mathcal{E}}(r)=\int_{0}^{\pi/2}\bigl(1-r^2\sin^2\theta\bigr)^{1/2}\td\theta,\\
{\mathcal{E}}'={\mathcal{E}}'(r)={\mathcal{E}}(r'),\\
{\mathcal{E}}(0)=\dfrac\pi2,\\
{\mathcal{E}}(1)=1
\end{cases}
\end{equation*}
respectively.
\par
For $0<r<1$, the formulas
\begin{gather*}
\frac{\td \mathcal{K}}{\td r}=\frac{\mathcal{E}-(r')^2{\mathcal{K}}}{r(r')^2},\quad
\frac{\td \mathcal{E}}{\td r}=\frac{\mathcal{E}-\mathcal{K}}{r},\quad
\frac{\td \big({\mathcal{E}}-(r')^2{\mathcal{K}}\bigr)}{\td r}=r{\mathcal{K}},\\
\frac{\td ({\mathcal{K}}-{\mathcal{E}})}{\td r}=\frac{r{\mathcal{E}}}{(r')^2},\quad
{\mathcal{E}}\biggl(\frac{2\sqrt{r}\,}{1+r}\biggr)=\frac{2\mathcal{E}-(r')^2\mathcal{K}}{1+r}
\end{gather*}
were presented in~\cite[Appendix~E, pp.~474--475]{Anderson-Vamanamurthy-Vuorinen-Wiley-conform-97}.

\begin{lemma}[{\cite[Theorem~3.21(1)~and~3.43 Exercise~13(a)]{Anderson-Vamanamurthy-Vuorinen-Wiley-conform-97}}]\label{lem2.1}
The function $\frac{\mathcal{E}-(r')^2\mathcal{K}}{r^2}$ is strictly increasing from $(0,1)$ onto $\bigl(\frac\pi4,1\bigr)$ and the function $2\mathcal{E}-(r')^2\mathcal{K}$ is increasing from $(0,1)$ onto $\bigl(\frac\pi2,2\bigr)$.
\end{lemma}

\begin{lemma}\label{lem2.2}
Let $u,\alpha\in(0,1)$ and
\begin{equation}\label{eq2.1}
 f_{u,\alpha}(r)=ur^2-(1-\alpha) \biggl\{\frac{2}{\pi}\bigl[2\mathcal{E}(r)-\bigl(1-r^2\bigr) \mathcal{K}(r)\bigr]-1\biggr\}.
\end{equation}
Then $f_{u,\alpha}>0$ for all $r\in(0,1)$ if and only if $u\ge (1-\alpha)\bigl(\frac4\pi-1\bigr)$ and $f_{u,\alpha}<0$ for all $r\in(0,1)$ if and only if $u\le\frac{1-\alpha}4$.
\end{lemma}

\begin{proof}
It is clear that
\begin{align}\label{eq2.2}
 f_{u,\alpha}\bigl(0^+\bigr)&=0,\\\label{eq2.3}
 f_{u,\alpha}(1^-)&=u-(1-\alpha)\biggl(\frac{4}{\pi}-1\biggr),\\\label{eq2.4}
 f_{u,\alpha}'(r)&=2r[u-(1-\alpha)g(r)],
\end{align}
where $g(r)=\frac{1}{\pi}\frac{\mathcal{E}-(r')^2\mathcal{K}}{r^2}$.
\par
When $u\geq\frac{1-\alpha}{\pi}$, from~\eqref{eq2.4} and Lemma~\ref{lem2.1} and by the monotonicity of $g(r)$, it follows that $f_{u,\alpha}(r)$ is strictly increasing on $(0,1)$. Therefore, $f_{u,\alpha}(r)>0$ for all $r\in(0,1)$.
\par
When $u\le \frac{1-\alpha}{4}$, from~\eqref{eq2.4} and Lemma~\ref{lem2.1} and by the monotonicity of $g(r)$, we obtain that $f_{u,\alpha}(r)$ is strictly decreasing on $(0,1)$. Therefore, $f_{u,\alpha}(r)<0$ for all $r\in(0,1)$.
\par
When $\frac{1-\alpha}{4}<u\le (1-\alpha)\bigl(\frac4\pi-1\bigr)$, from~\eqref{eq2.3} and~\eqref{eq2.4} and by the monotonicity of $g(r)$, we see that there exists $\lambda\in(0,1)$ such that $f_{u,\alpha}(r)$ is strictly increasing in $(0,\lambda]$ and strictly decreasing in $[\lambda,1)$ and
\begin{equation}\label{eq2.5}
 f_{u,\alpha}(1^-)\le 0.
\end{equation}
Therefore, making use of the equation~\eqref{eq2.2}, the inequality~\eqref{eq2.5}, and the piecewise monotonicity of $f_{u,\alpha}(r)$ lead to the conclusion that there exists $0<\lambda< \eta < 1$ such that $f_{u,\alpha}(r)>0$ for $r\in(0,\eta)$ and $f_{u,\alpha}(r) < 0$ for $r\in (\eta,1)$.
\par
When $(1-\alpha)\bigl(\frac4\pi-1\bigr)\le u<\frac{1-\alpha}\pi$, by~\eqref{eq2.3}, it follows that
\begin{equation}\label{eq2.6}
 f_{u,\alpha}(1^-)\ge 0.
\end{equation}
\par
From~\eqref{eq2.3} and~\eqref{eq2.4} and by the monotonicity of $g(r)$, we see that there exists $\lambda\in(0,1)$ such that $f_{u,\alpha}(r)$ is strictly increasing in $(0,\lambda]$ and strictly decreasing in $[\lambda,1)$.
Therefore, $f_{u,\alpha}(r)>0$ for $r\in(0,1)$ follows from~\eqref{eq2.2} and~\eqref{eq2.6} together with the piecewise monotonicity of $f_{u,\alpha}(r)$.
\end{proof}

\section{Main Results}
Now we are in a position to state and prove our main results.

\begin{theorem}\label{th3.1}
If $\alpha\in(0,1)$ and $\lambda ,\mu \in\bigl(\frac12,1\bigr)$, then the double inequality
\begin{multline}\label{eq3.1}
 C(\lambda a+(1-\lambda )b,\lambda b+(1-\lambda )a)<\alpha A(a,b)+(1-\alpha)T(a,b)\\
 < C(\mu a+(1-\mu )b,\mu b+(1-\mu )a)
\end{multline}
holds for all $a,b>0$ with $a\ne b$ if and only if
\begin{equation*}
\lambda \le \frac{1}{2}+\frac{\sqrt{1-\alpha}\,}{4} \quad \text{and}\quad
\mu \ge \frac{1}{2}\Biggl[1+\sqrt{(1-\alpha)\biggl(\frac4\pi-1\biggr)}\,\Biggr].
\end{equation*}
\end{theorem}

\begin{proof}
Since $A(a,b)$, $T(a,b)$, and $C(a,b)$ are symmetric and homogeneous of degree one, without loss of generality, assume that $a>b$. Let $p\in\bigl(\frac12,1\bigr)$, $t=\frac{b}a\in(0,1)$, and $r=\frac{1-t}{1+t}$. Then
\begin{align*}
 & \quad C(pa+(1-p)b,pb+(1-p)a)-\alpha A(a,b)-(1-\alpha)T(a,b)\\
 &=a\frac{[p+(1-p){b}/{a}]^2+(p{b}/{a}+1-p)^2}{1+{b}/{a}}-\alpha a\frac{1+{b}/{a}}{2} -(1-\alpha)\frac{2a}{\pi} \mathcal{E}\Biggl(\sqrt{1-\biggl(\frac{b}a\biggr)^2}\,\Biggr)\\
 &=a\biggl\{\frac{[p+(1-p)t]^2+(pt+1-p)^2}{1+t}-\alpha\frac{1+t}{2} -(1-\alpha)\frac{2}{\pi}\mathcal{E}\Bigl(\sqrt{1-t^2}\,\Bigr)\biggr\}\\
 &=a\biggl\{\frac{(1-2p)^2r^2+1}{1+r}-\alpha\frac{1}{1+r} -(1-\alpha)\frac{2}{\pi}\frac{2\mathcal{E}-(r')^2\mathcal{K}}{1+r}\biggr\}\\
 &=\frac{a}{1+r}\biggl[(1-2p)^2r^2+1-\alpha -(1-\alpha)\frac2\pi\bigl(2\mathcal{E}-(r')^2\mathcal{K}\bigr)\biggr].
\end{align*}
From this and Lemma~\ref{lem2.2}, Theorem~\ref{th3.1} follows.
\end{proof}

\begin{corollary}
For $r\in(0,1)$ and $r'=\sqrt{1-r^2}\,$, we have
\begin{equation}\label{eq3.3}
\frac{\pi}{2}\bigg[\frac{17+30r'+17(r')^{2}}{8(1+r')}-\frac{3(1+r')}{2}\bigg]<\mathcal{E}(r)
<{\pi}\bigg[\frac{r'+{2}(1-r')^2/{\pi}}{1+r'}\bigg].
\end{equation}
\end{corollary}

\begin{proof}
This follows from letting $\alpha=\frac{3}{4}$, $\lambda=\frac{5}{8}$, and $\mu=\frac{1}{2}\Bigl(1+\frac{\sqrt{4/\pi-1}\,}{2}\Bigr)$ in Theorem~\ref{th3.1}.
\end{proof}

\section{Remarks}

\begin{remark}
Recently, the complete elliptic integrals have attracted attention of numerous mathematicians. In~\cite{Chu-Wang-Qiu-Indian-Acad-12}, it was established that
\begin{multline}\label{eq3.4}
\frac{\pi}{2}\bigg[\frac{1}{2}\sqrt{\frac{1+(r')^2}{2}}\,+\frac{1+r'}{4}\bigg]<\mathcal{E}(r)\\
<\frac{\pi}{2}\bigg[\frac{4-\pi}{\bigl(\sqrt{2}\,-1\bigr)\pi}\sqrt{\frac{1+(r')^2}{2}}\, +\frac{\bigl(\sqrt{2}\,\pi-4\bigr)(1+r')}{2\bigl(\sqrt{2}\,-1\bigr)\pi}\bigg],
\end{multline}
for all $r\in(0,1)$.
In~\cite{Guo-Qi-MIA-11} it was proved that
\begin{equation}\label{eq3.5}
\frac{\pi}{2}-\frac{1}{2}\log{\frac{(1+r)^{1-r}}{(1-r)^{1+r}}}<\mathcal{E}(r)
<\frac{\pi-1}{2}+\frac{1-r^2}{4r}\log{\frac{1+r}{1-r}},
\end{equation}
for all $r\in(0,1)$.
In~\cite{Yin-Qi-Ellip-Int.tex} it was presented that
\begin{equation}\label{eq3.6}
\frac{\pi}{2}\frac{\sqrt{6+2\sqrt{1-r^{2}}-3r^{2}}\,}{2\sqrt{2}\,}\le {\mathcal{E}}(r)\le \frac{\pi}{2}\frac{\sqrt{10-2\sqrt{1-r^{2}}-5r^{2}}\,}{2\sqrt{2}\,}
\end{equation}
for all $r\in(0,1)$.
In~\cite{Chu-Wang-Qiu-Indian-Acad-12} it was pointed out that the bounds in~\eqref{eq3.4} for $\mathcal{E}(r)$ are better than the bounds in~\eqref{eq3.5} for some $r\in (0,1)$.
\end{remark}

\begin{remark}
The lower bound in~\eqref{eq3.3} for $\mathcal{E}(r)$ is better than the lower bound in~\eqref{eq3.4}. Indeed,
\begin{multline*}
\frac{17+30x+17x^{2}}{8(1+x)}-\frac{3(1+x)}{2} -\bigg[\frac{1}{2}\sqrt{\frac{1+x^2}{2}}\,+\frac{1+x}{4}\bigg]\\
=\frac{3x^2+2x+3-2\sqrt{2(1+x^2)}\,(1+x)}{8(1+x)}
\end{multline*}
and
$$
(3x^2+2x+3)^2-\bigl[2\sqrt{2(1+x^2)}\,(1+x)\bigr]^2=(1-x)^4>0
$$
for all $x\in (0,1)$.
\end{remark}

\begin{remark}
The following equivalence relations show that the lower bound in~\eqref{eq3.3} for $\mathcal{E}(r)$ is better than the lower bound in~\eqref{eq3.6}:
\begin{gather*}
\frac{17+30x+17x^{2}}{8(1+x)}-\frac{3(1+x)}{2}>\frac{\sqrt{6+2x-3(1-x^2)}\,}{2\sqrt{2}\,}\\
\Longleftrightarrow(5x^2+6x+5)^2>8(x+1)^2(3x^2+2x+3)\\
\Longleftrightarrow(x-1)^4>0,
\end{gather*}
where $x\in (0,1)$.
\end{remark}

\end{document}